\documentclass[11pt]{article}
\usepackage{amsfonts,amsmath,latexsym,amsthm,amssymb}
\usepackage{dsfont}
\usepackage{enumerate}
\usepackage{mathrsfs}
\usepackage{listings}
\usepackage{amscd}

\input epsf.tex
\def\<{\left<}
\def\>{\right>}
\def\({\left(}
\def\){\right)}

\def\9{\infty}
\def\R{\mathbb R}

\def\E{\mathbb E}

\newtheorem{theorem}{Theorem}

\newtheorem{corollary}[theorem]{Corollary}

\newtheorem{definition}[theorem]{Definition}

\newtheorem{lemma}[theorem]{Lemma}

\newtheorem{proposition}[theorem]{Proposition}
\newtheorem{remark}[theorem]{Remark}

\newcommand{\beqnar}{\begin{eqnarray*}}
\newcommand{\eeqnar}{\end{eqnarray*}}
\newcommand{\ba}{\begin{array}}
\newcommand{\ea}{\end{array}}

\begin{document}

\title{Probabilistic representation for solutions of a porous media type 
equation  with Neumann boundary condition: the case of the half-line}
\author{ Ioana Ciotir (1) and Francesco Russo (2) }
\date{April 11th 2013}
\maketitle

\thispagestyle{myheadings}

\textbf{Summary:} The purpose of this paper consists in proposing a generalized solution
for a porous media type equation on a half-line with Neumann boundary
condition and prove a probabilistic representation of this solution in terms
of an associated microscopic diffusion. The main idea is to construct a
stochastic differential equation with reflection which has a solution in law
and whose marginal law densities provide the unique solution of the porous media type
equation.


\textbf{2000 AMS-classification}: 60H10, 60H30, 60H10, 60G46, 35C99.

\begin{itemize}
\item[(1)] Ioana Ciotir, University A1.I. Cuza, Ro--6600 Iasi, Romania and 
\newline
University of Neuch\^atel, Switzerland

\item[(2)] Francesco Russo, ENSTA ParisTech, Unit\'e de Math\'ematiques
appliqu\'ees, France.
\end{itemize}

\vfill \eject

\section{Introduction and preliminaries}

\bigskip

In this work we focus on
  a porous media type equation given by%
\begin{equation}
\left\{ 
\begin{array}{ll}
\partial _{t}u(t,x) \in \frac{1}{2}\partial _{xx}^{2}\beta \left( u\right) (t,x)
& :(t,x)\in \left( 0,T\right] \times \mathbb{R}_{+},\medskip \\ 
u(0,x)=u_{0}(x) & :x\in \mathbb{R}_{+},\medskip \\ 
\partial _{x}\left( \beta \left( u\right) \right) \left( t,0\right) =0, & 
:t\in \left( 0,T\right]. 
\end{array}
\right.  \label{equ-half}
\end{equation}
The natural analytical concept of (weak) solution of \eqref{equ-half}
is given in Definition \ref{DWeak} and it involves the restriction of 
the derivative of $\beta(u)$ on the boundary. 
We introduce here  a new notion of solution that we call 
{\it  generalized  solution}
for \eqref{equ-half}, which do not require, a priori, the existence of
distributional  derivatives for $\beta(u)$.

Under some minimal conditions, we will first concentrate on uniqueness of
the generalized solutions (in a large class) and existence of a weak solution
(smaller class).
In particular, we will include the case when
$\beta$ is possibly discontinuous.
 Moreover   we  are interested in its 
probabilistic representation through the marginal laws of a stochastic process.

We formulate now some assumptions.

\bigskip

\textbf{Assumption 1}

\begin{description}
\item{i)} $u_{0}\in L^{1}\left( \mathbb{R}_{+}\right) \cap L^{\infty }\left( 
\mathbb{R}_{+}\right) $ is an initial probability density;

\item{ii)} $\beta :\mathbb{R}_+ \rightarrow \mathbb{R}_+$ is a monotone 
increasing function
with $\beta \left( 0\right) =0.$

\item{iii)} There is a constant $c$ such that $\left\vert \beta \left( u\right)
\right\vert \leq cu.$
\end{description}

\bigskip


With $\beta $ we naturally associate a maximal monotone graph still denoted by
the same letter $\beta :\mathbb{R}_+ \rightarrow 2^{\mathbb{R}_+}$, by filling
the gaps, i.e. by identifying $\beta (x)$ with the interval $[\beta
(x-),\beta (x+)]$. We consider now $\Phi :\mathbb{R}_+^{\ast }\rightarrow 
\mathbb{R}_{+}$ such that $\beta (u)=\Phi ^{2}(u) u,  u\neq 0$. 
Again we associate naturally $\Phi$ with the non-negative graph 
(still denoted by $\Phi$) 
 $ u \mapsto \sqrt{\frac{\beta(u)}{u}}$.
Finally we extend the graph $\Phi$ it to $\mathbb{R}$, defining
$\Phi(0): = [\liminf_{x\rightarrow 0+}\Phi (x),\limsup_{x\rightarrow 0+}\Phi (x)]$.
Of course, if $\beta$ is continuous, the first line of 
\eqref{equ-half} can be replaced by the most natural equality
$\partial _{t}u(t,x) = \frac{1}{2}\partial _{xx}^{2}\beta \left( u\right) (t,x)$.

\bigskip

\begin{definition} \label{DND}
\begin{itemize}
\item[i)] We say that $\beta $ is \textbf{non-degenerate} if there is a
constant $c_{0}>0$ such that $\Phi \geq c_{0},$ for every $x > 0$.

\item[ii)] We say that $\beta $ is degenerate if $\underset{u\rightarrow 0_+}{%
\lim }\Phi \left( u\right) =0.$

\item[iii)] We say that $\beta $ is strictly increasing after some zero $%
u_{c}$ if there is $u_{c}\geq 0$ such that $\beta \left\vert _{\left[ 0,u_{c}%
\right[ }\right. =0$ and $\beta \left\vert _{\left[ u_{c},\infty \right[
}\right. $ is strictly increasing.
\end{itemize}
\end{definition}

\bigskip


Often also the Assumption 2 below will be in force.

\textbf{Assumption 2.} We suppose that one of the following properties is
verified.

\begin{description}
\item{i)} $\beta $ is non-degenerate.

\item{ii)} $\beta $ is degenerate and there is a discrete ordered set of elements $(\{e_k\}$ of  $\R_+ \cup \{ +\infty \}$) so 
that $u_0$ has locally bounded variation on $)e_k, e_{k+1}($.
\item{iii)} $\beta $ is strictly increasing after some zero $u_{c}$ 
(in particular
it is degenerate).
\end{description}

\bigskip


As we anticipated,  the idea is to construct a stochastic process $Y$ such that 
the  (marginal) law of $Y_t$
has a density given by $u(t,\cdot )$ for any $t \in [0,T]$. 
We look for $Y$ as
being a solution (in law) of the stochastic differential equation
 with reflection%
\begin{equation}
\left\{ 
\begin{array}{l}
dX_{t}=\Phi \left( u\left( t,X_{t}\right) \right) dB_{t}+dK_{t},  \\ 
K\ \text{increasing\ process}\int_{0}^{T}Y_{s}dK_{s}=0, \\ 
u\left( s,.\right) =\text{law density of }Y_{t},%
\end{array}%
\right.  \label{SDE}
\end{equation}%
 which has a weak
solution $X$ whose law density is the unique solution of the (\ref{equ-half}).

As far as our knowledge is concerned, this paper could be the first 
one studying
the probabilistic representation of a non-linear partial differential 
equation (PDE)
 with Neumann boundary conditions on some domain, through a reflected non-linear diffusion. 

The problem of probabilistic representation related to a PDE, related
to a class ${\mathcal A}$ of solutions,
 is the following.
For each of $ u \in {\mathcal A}$,
there exists 
a stochastic process, solving some form of stochastic differential equation
whose coefficients involve the law of the process (McKean-Vlasov type),
whose marginal laws are given by $u$. 
Solutions of those stochastic differential equations are also called non-linear 
diffusions. The PDE is intended
as a non-linear forward Kolmogorov's equation 
corresponding to the non-linear diffusion.

The paper provides a bridge between two big areas of stochastic analysis:
 stochastic differential equations with reflection on some domain, non-linear diffusions on the whole line.
As we will see the literature is rich of contributions in both topics,
but in principle no one connects them.

\begin{enumerate}
\item \textbf{Non-linear diffusion problems} \\
There are several contributions to the study of  equations stated in the first line of (\ref{equ-half})
but on the whole line or even on $\R^d$. That equation, which  will be precisely stated in
(\ref{equ-all}), was first investigated  by \cite{BeBrC75} for existence, \cite{BrC79} for uniqueness and \cite{BeC81}
 for continuous dependence on coefficients. \\
The physical interpretation of the probabilistic representation is the 
following. 
The singular non-linear diffusion equation \eqref{equ-all} describes a \textit{macroscopic} phenomenon for which the probabilistic 
representation tries to give a \textit{microscopic} probabilistic interpretation via a non-linear stochastic diffusion
 equation modeling the evolution of a single point on the layer.\\
To our knowledge, the first author who considered a probabilistic representation for the solution of
 a non-linear deterministic PDE (on the whole line),
 was McKean \cite{mckean}, especially in relation to the so called
 propagation of chaos. He supposed to have smooth coefficients. 
After that, the literature grew and nowadays there is a vast amount of contributions to the subject,
 particularly when the non-linearity $\beta(u)$ appears inside  the first order part, as e.g. in Burgers equations
 (see for instance the survey papers \cite{graham} and \cite{sznit}).\\
A probabilistic interpretation of \eqref{equ-all} when $\beta(u)=\vert u\vert u^{m-1}$, for $ m>1 $ was provided for instance in 
\cite{vallois} in which probabilistic representations of the
 Barenblatt solutions and of a large class of solutions were given.
Later, when $\Phi$ is of  class $C^3$, Lipschitz, $\beta$ is non-degenerate
and $u_0$ is smooth enough,
\cite{jourdainmeleard} provided also strong solutions 
to the probabilistic representation problem,
see more precisely Remark \ref{R12}. 
In particular,  the probabilistic representation of the porous media on $\mathbb{R}$ was studied in the case of 
irregular coefficients in   \cite{BRR1, BRR2} 
with refinements in \cite{BR, BCR2, BCR3}. In particular \cite{BRR1} represented when $\beta$ is non-degenerated
all the solutions in the sense of distributions under Assumption 1. 
Moreover also the uniqueness of the corresponding 
non-linear diffusions was established.
When $\beta$ is degenerate, under Assumptions 1. and 2. ii) or 2. iii), \cite{BRR2} has provided again
probabilistic representations, but not uniqueness of those. Some improvements also appeared
in   \cite{BCR2, BCR3}, at least when $\Phi$ is continuous.
 \cite{BCR2} provides probabilistic representation
of the Barenblatt solutions when $\beta(u) = u^m, 
\frac{3}{5} < m < 1$, i.e. in the case of fast diffusions.

\item \textbf{Stochastic differential equations with reflection}. There is
a vast literature in the subject, in the one-dimensional case, and in the
multidimensional case as well. It is for us impossible to quote all those.
In the half-line case $\mathbb{R}_+$, such an equation can be formulated as
follows: 
\begin{equation*}
dX_t = \sigma(t, X_t) dW_t + b(t,X_t) dt + dK_t,
\end{equation*}
where $\sigma, b:[0,T] \times \mathbb{R}$ are Borel functions and $K$ is an
increasing process such that $\int_0^T X_s dK_s = 0$. The solution is the
couple $(X,K)$. When $\sigma$ and $b$ are Lipschitz the theory is
well-developed, see for instance \cite{lionssznitman},  \cite{cepa} and \cite{slominski}  (and references therein).
All the three papers
have treated the case when the coefficients are time-independent.
 Elements related to the time-dependent case
appear for instance in \cite{chaleyat} and in  \cite{revuz-yor},
 chapter   IX, Exercise (2.14), page
385. \\ 
When the coefficients are non-Lipschitz, but are Lipschitz with a logarithmic correction and the domain is a half-line \cite{BOWANG},
proved recently  the existence of a strong solution.

As far as weak (in law) solutions of SDEs with reflection, the pioneering work is
\cite{stroockreflection}, which solves some submartingale problem related to 
stochastic differential equations with reflection under relatively general
conditions on the (even time-dependent) coefficients.
An interesting  work is also \cite{pardouxwilliams}  which constructed solutions of
symmetric (time-homogeneous) diffusions with reflection on a domain of $\R^n$,
via Dirichlet forms.

Reflected diffusions  are naturally candidates  for the probabilistic representation
of  solutions of (linear) PDE with Neumann boundary condition in the following sense.
Given a   specific solution $u$ to a Fokker-Planck type PDE, with Neumann boundary conditions,
a process $X$ represents it probabilistically if the marginal laws of $X$
are solutions of  the   PDE. As far as we know, even in this linear case, 
the point is not clear in the literature. For instance, by applying some It\^o formula type,
it is possible to show that solutions of SDEs (or martingale problems)
solve a PDE with boundary condition, in some sense. However
to show that a given solution to a Fokker-Planck PDE can be represented through a 
process is rarely explained. This is related to the
study of uniqueness of the mentioned PDE, see for instance \cite{BRR1},
 Theorem 3.8,
for an equation in the whole line.\\
We conclude this discussion about probabilistic representation of linear PDEs mentioning
\cite{BarbuNeumann}, which gives a representation  of the solution
of an elliptic problem with unbounded monotone drift in term of the invariant measure
of a reflection diffusion equation considered by \cite{cepa}.
\end{enumerate}

\bigskip
We are aware that the  Neumann problem on
the half-line constitutes somehow a toy model, however
at our knowledge, there are no results in the literature
about well-statement of that problem.
If $x$ varies in a bounded domain (for instance a compact interval), there are some contributions
at least in the case when  $\beta(u) = u ^m, 0 < m < 1$, which constitutes
the case of the classical porous media equation.
Given an integrable initial condition $u_0$, \cite{alikakos}, through the techniques of maximal accretivity, see e.g. \cite{showalter},
\cite{barbu10}, a $C^0$-type solution (or mild solution), see Chapter IV.8 of
\cite{showalter}. The technique consists in showing that the elliptic
corresponding operator is m-accretive, whose step was performed
by \cite{brezisstrauss}. In Corollary 3.5, \cite{alikakos} shows that 
(when $u_0 \in L^\infty$), that the solution is even classical and it is therefore
a weak solution in  the sense of 
Definition \ref{DWeak}, adapted to the case when an interval replaces the real line.


The paper is organized as follows. After the introduction above, 
at Section \ref{SPrelim} we introduce the basic definition of
solutions and some notations.
At Section \ref{S2}, we discuss existence and uniqueness
of \eqref{equ-half} and we remark that the solutions 
have some minimal regularity properties.
Section \ref{S3} is devoted to the existence of the probabilistic 
representation in the form of a solution to  a non-linear 
(in the sense that the marginal densities appear
in the coefficients)
stochastic differential equation, with reflection.
The Appendix is devoted to the equivalence of weak and generalized
solutions under some minimal regularity conditions.
\bigskip

\section{Preliminaries}

\label{SPrelim}

\bigskip Let $I$ be a real interval.
Given a function $\varphi: [0,T] \times I \rightarrow \R$,
$(t,x) \mapsto \varphi(t,x)$,
we denote (if it exists), by $\varphi'$ (resp. $\varphi''$)
the partial derivative  $\partial_x \varphi$ (resp. second partial derivative  $\partial^2_{xx} \varphi$)   
with respect to the second argument,
 defined again on $[0,T] \times I$.
If $I$ is closed, then the derivatives are defined as continuous extensions
from $[0,T] \times {\rm Int} I$.

In this paper $C_{0}^{\infty }\left( \mathbb{R}_{+}\right)$ will
denote the set of functions $\varphi: \mathbb{R}_+ \rightarrow \mathbb{R}$
which are restrictions of smooth functions with compact support defined on $%
\mathbb{R}$. 
We denote by $W^{1,1}_{\rm loc}(\R_+)$  the space of 
 absolutely continuous functions $f: \R_+ \rightarrow \R$
such that  for every compact subset $K$ of $\R_+$
$\int_{K} \vert f'(x) \vert dx $ is finite.
 Of course for any $f \in W^{1,1}_{\rm loc}(\R_+)$ and every compact $K$
of $\R$ we have $\int_{K} \vert f(x) \vert dx < \infty $.

If $I = \R$ or $\R_+$, we denote a bit abusively by $ W^{1,1}_{\rm loc} 
([0,T] \times I)$
the set of $f: [0,T] \times \R \rightarrow I$ such that
for almost all $t$, we have
$f(t, \cdot)  \in W^{1,1}_{\rm loc}(I)  $ and 
for $\int_{[0,T] \times K} \vert \partial_x f(t, x)  \vert < \infty,$
for every compact $K$ of $I$.

\begin{definition}
\label{def1} A function $u\in L^{1}\left( \left[ 0,T\right] \times \mathbb{R}%
_{+}\right) $ is called  \textbf{generalized solution} for equation (\ref{equ-half}) if, for any $\varphi \in C_{0}^{\infty }\left( \mathbb{R}%
_{+}\right) $ such that $\varphi ^{\prime }\left( 0\right) =0$, we have 
\begin{equation}
\int_{0}^{\infty }\varphi \left( x\right) u\left( t,x\right)
dx=\int_{0}^{\infty }\varphi \left( x\right) u_{0}\left( x\right) dx+\frac{1%
}{2}\int_{0}^{t}\int_{0}^{\infty }\varphi ^{\prime \prime }\left( x\right)
\eta _{u}\left( s,x\right) dxds,  \label{def-half}
\end{equation}%
where $\eta _{u}:\left[ 0,T\right] \times \mathbb{R}_{+}\rightarrow \mathbb{R%
}$, $\eta _{u}\in L^{1}\left( \left[ 0,T\right] \times \mathbb{R}_{+}\right) 
$ is such that 
\begin{equation*}
\eta _{u}\left( t,x\right) \in \beta \left( u\left( t,x\right) \right)
,\quad dt\otimes dx-a.e.~\left( t,x\right) \in \left[ 0,T\right] \times 
\mathbb{R}_{+}.
\end{equation*}
\end{definition}

We remark a generalized solution in the same spirit, for porous media
equations but with Dirichlet boundary conditions was given in [\cite{barbu10}
page 226]). 
Moreover, in our case the test function is time independent. 

\bigskip

\begin{remark} \label{R3bis}
The formal justification of  previous formula comes out from the following
observation. Suppose that $\left(u,\eta _{u}\right) $ is a smooth solution
to (\ref{equ-half}) and $u_{0}$ is continuous. Then, for every $x\in 
\mathbb{R}_{+}$, we have
\begin{equation*}
u(t,x)=u_{0}(x)+\frac{1}{2}\int_{0}^{t}\eta _{u}^{\prime \prime }\left(
s,x\right) ds.
\end{equation*}%
Let $\varphi $ be a test function from $C_{0}^{\infty }\left( \mathbb{R}%
_{+}\right) $ such that $\varphi ^{\prime }\left( 0\right) =0.$ In this
particular situation we have%
\begin{eqnarray*}
\int_{0}^{\infty }\varphi \left( x\right) u\left( t,x\right) dx
&=&\int_{0}^{\infty }\varphi \left( x\right) u_{0}\left( x\right) dx+\frac{1%
}{2}\int_{0}^{t}\int_{0}^{\infty }\varphi \left( x\right) \eta _{u}^{\prime
\prime }\left( s,x\right) dxds \\
&=&\int_{0}^{\infty }\varphi \left( x\right) u_{0}\left( x\right) dx+\frac{1%
}{2}\int_{0}^{t}\varphi \left( 0\right) \eta _{u}^{\prime }\left( s,0\right)
ds \\
&&\quad \quad \quad \quad \quad \quad \quad -\frac{1}{2}\int_{0}^{t}%
\int_{0}^{\infty }\varphi ^{\prime }\left( x\right) \eta _{u}^{\prime
}\left( s,x\right) dxds \\
&=&\int_{0}^{\infty }\varphi \left( x\right) u_{0}\left( x\right) dx+\frac{1%
}{2}\int_{0}^{t}\varphi \left( 0\right) \eta _{u}^{\prime }\left( s,0\right)
ds \\
&&+\frac{1}{2}\int_{0}^{t}\int_{0}^{\infty }\varphi ^{\prime \prime }\left(
x\right) \eta _{u}\left( s,x\right) dxds-\frac{1}{2} 
\underbrace{\int_{0}^{t}\varphi
^{\prime }\left( 0\right) \eta _{u}\left( s,0\right) ds}_{= 0},
\end{eqnarray*}%
which constitutes indeed \eqref{def-half}. 
\end{remark}
As we mentioned in the introduction, the natural analytical concept
of {\bf solution} should involve the first spatial  derivative at the boundary.\\

We consider a couple $\left( u,\eta _{u}\right) $ satisfying  Definition 
\ref{def1} and suppose that   $\eta_u \in W_{loc}^{1,1}([0,T] \times 
\mathbb{R}_{+})$ in agreement with the notations of Section \ref{SPrelim}.
In particular 
  for every compact $ K \in \mathbb{R} $ we have  
 \begin{equation} \label{eeweak}
 \int_{\left[ 0,T\right] \times K}\left\vert \eta _{u}^{\prime }\left(
 s,x\right) \right\vert dsdx<\infty .
 \end{equation}

\begin{definition} \label{DWeak}
The couple $\left( u,\eta _{u}\right) $ is said to be a weak solution of 
\eqref{equ-half} 
if for every $ \varphi \in C_{0}^{\infty}(\mathbb{R}_{+}) $ we have 
\begin{equation} \label{weak}
\int_{0}^{\infty }\varphi \left( x\right) u\left( t,x\right)
dx=\int_{0}^{\infty }\varphi \left( x\right) u_{0}\left( x\right)
dx-\int_{0}^{t}\int_{0}^{\infty }\varphi ^{\prime }\left( x\right) \eta
_{u}^{\prime }\left( s,x\right) dxds.
\end{equation}
\end{definition}
\begin{proposition} \label{PWeakDistri}
Let $\left( u,\eta _{u}\right) $ be a couple such that 
$\eta_{u}(t,x) \in \beta(u(t,x)) $ and
 $ \eta_{u}
\in  W^{1,1}_{loc}([0,T] \times \mathbb{R}_{+}) $.
Then $ u $ is a generalized solution if and only if it is a weak solution in the sense of Definition \ref{DWeak}.
\end{proposition}
\begin{proof}
See Appendix.
\end{proof}


\section{The porous media equation on half-line with Neumann boundary
condition}

\label{S2} \bigskip

In this part of the paper we will study existence and uniqueness of the
generalized solution for equation \eqref{equ-half}. We also show the connection
with the notion of weak solutions.

This will be done using the known results on the whole line $\mathbb{R}$. 
To this purpose, we start extending the initial condition to the real line by the following construction. 
\newline
Let $\overline{u_{0}}\in (L^{1}\cap L^{\infty }) \left( \mathbb{R}\right) $ be
defined by 
\begin{equation}
\overline{u_{0}}\left( x\right) =\left\{ 
\begin{array}{cc}
\frac{1}{2}u_{0}\left( x\right) & ,x\geq 0 \\ 
\frac{1}{2}u_{0}\left( -x\right) & ,x<0.%
\end{array}%
\right.  \label{uzerobar}
\end{equation}%
and $\overline{\beta }:\mathbb{R\rightarrow R}$ by 
\begin{equation*}
\overline{\beta }\left( u\right) =\frac{1}{2}\beta \left( 2u\right) ,\quad
u\in \mathbb{R}\text{.}
\end{equation*}

We can now consider the corresponding porous media equation on the whole line%
\begin{equation}
\left\{ 
\begin{array}{ll}
\partial _{t}\overline{u}(t,x)=\frac{1}{2}\partial _{xx}^{2}\overline{\beta }%
\left( \overline{u}\right) (t,x) & :(t,x)\in \left( 0,T\right] \times 
\mathbb{R},\medskip \\ 
\overline{u}(0,.)=\overline{u}_{0}, & 
\end{array}%
\right.  \label{equ-all}
\end{equation}%
which, by Proposition 3.4 from \cite{BRR1} (see also 
\cite{BeC81})  has a unique solution in the
sense of distributions, i.e. there exists a unique couple $\left( \overline{u%
},\eta_{\bar u}\right) \in \left( L^{1}\cap L^{\infty }\right) \left( %
\left[ 0,T\right] \times \mathbb{R}\right) $ such that 
\begin{equation}
\int_{\mathbb{R}}\varphi \left( x\right) \overline{u}\left( t,x\right)
dx=\int_{\mathbb{R}}\varphi \left( x\right) \overline{u}_{0}\left( x\right)
dx+\frac{1}{2}\int_{0}^{t}\int_{\mathbb{R}}\varphi ^{\prime \prime }\left(
x\right)\eta_{\bar u}\left( s,x\right) dxds,  \label{def-all}
\end{equation}%
for all $\varphi \in C_{0}^{\infty }\left( \mathbb{R}\right) $ and 
\begin{equation*}
\eta_{\bar u}\left( t,x\right) \in \overline{\beta }\left( \overline{u%
}\left( t,x\right) \right) ,\text{ for }dt\otimes dx-~a.e.~\left( t,x\right)
\in \left[ 0,T\right] \times \mathbb{R}.
\end{equation*}


\begin{remark}
\label{R3} Since $\overline{u_{0}}$\ is even, the solution $\overline{u}$ to
equation (\ref{equ-all}) and ${\eta}_{\bar u}$ are also even. In fact, by
applying \eqref{def-all} to $x\mapsto \varphi \left( -x\right) $ and making
then the change of variables $x\mapsto -x$, we show that $(t,x)\mapsto 
\overline{u}\left( t,-x\right) $ is also a solution. 
The result follows by uniqueness of \eqref{equ-all}. 
\end{remark}

\begin{proposition}
\label{P4} We define $v, \eta_v:[0,T] \times \R_+$, setting $v(t,x)=2\overline{u}\left( t,x\right) $ and $\eta
_{v}(t,x) = 2 \eta_{\bar u}(t,x), \forall (t,x) \in [0,T] \times \R_+.$
The couple $\left( v,\eta _{v}\right) $ is a
generalized solution to equation (\ref{equ-half}) in the sense of Definition \ref{def1}.
\end{proposition}

\begin{proof} 
Since 
\begin{equation*}
2\overline{\beta }\left( \overline{u}\right) =\beta \left( 2\overline{u}%
\right) =\beta \left( v\right), 
\end{equation*}
we first observe that $\eta _{v}(t,x) \in \beta(v(t,x)) dt dx$ a.e.
Moreover $\eta _{v}$ belongs to $L^1([0,T] \times \R_+)$ since 
 $\eta _{\bar u}$ belongs to $L^1([0,T] \times \R)$. It remains to prove 
that
 $\left( v,\eta_v \right)$
satisfies (\ref{def-half})
for all $\varphi \in C_{0}^{\infty }\left( \mathbb{R}_{+}\right) $ such that 
$\varphi ^{\prime }\left( 0\right) =0.$

\bigskip

\textit{Step I}

First we prove that (\ref{def-half}) is true for all test functions $\varphi
\in C_{0}^{\infty }\left( \mathbb{R}_{+}\right) $ with $\varphi ^{\prime
}\left( 0\right) =0$ and which extend to an even smooth function $\bar
\varphi$ with compact support on $\mathbb{R}$. 
By (\ref{def-all}) we have 
\begin{equation*}
\int_{\mathbb{R}}\overline{\varphi }\left( x\right) \overline{u}\left(
t,x\right) dx=\int_{\mathbb{R}}\overline{\varphi }\left( x\right) \overline{u%
}_{0}\left( x\right) dx+\frac{1}{2}\int_{0}^{t}\int_{\mathbb{R}}\overline{%
\varphi }^{\prime \prime }\left( x\right)\eta_{\bar u}\left(
s,x\right) dxds.
\end{equation*}

Since $\overline{\varphi },$ $\overline{u}_{0}$ and $\overline{u}(t,\cdot)$, for any $t \in [0,T]$ are even
functions, we get, for the first two terms, that 
\begin{equation*}
\int_{\mathbb{R}}\overline{\varphi }\left( x\right) \overline{u}\left(
t,x\right) dx=\int_{0}^{\infty } \varphi \left( x\right) 2%
\overline{u}\left( t,x\right) dx=\int_{0}^{\infty } \varphi \left(
x\right) v\left( t,x\right) dx,
\end{equation*}%
\begin{equation*}
\int_{\mathbb{R}}\overline{\varphi }\left( x\right) \overline{u}_{0}\left(
x\right) dx=\int_{0}^{\infty }\overline{\varphi }\left( x\right) 2\overline{u%
}_{0}\left( x\right) dx=\int_{0}^{\infty }\overline{\varphi }\left( x\right)
u_{0}\left( x\right) dx.
\end{equation*}%
Since $\eta_v = 2 \eta_{\bar u}$
we  obtain%
\begin{eqnarray*}
\int_{\mathbb{R}} \bar \varphi^{\prime \prime }\left( x\right) 
\eta_{\bar u}\left( s,x\right) dx &=&\int_{0}^{\infty }\overline{%
\varphi }^{\prime \prime }\left( x\right) 2 \eta_{\bar u}\left(
s,x\right) dx \\
&=&\int_{0}^{\infty } \varphi^{\prime \prime }\left( x\right)
\eta _{v}\left( s,x\right) dx,\ s\in \lbrack 0,T].
\end{eqnarray*}%
This proves \eqref{equ-half} for the restricted class of $\varphi $ which
extend to an even smooth function on $\mathbb{R}$ with compact support.

\bigskip

\textit{Step II}

Now we can prove the general statement.

Let $\varphi \in C_{0}^{\infty }\left( \mathbb{R}_{+}\right) $ such that $%
\varphi ^{\prime }\left( 0\right) =0.$ We extend it to an even function 
\begin{equation*}
\overline{\varphi }\left( x\right) =\left\{ 
\begin{array}{cc}
\varphi \left( x\right) & : x\geq 0 \\ 
\varphi \left( -x\right) & : x<0,%
\end{array}%
\right.
\end{equation*}%
which has compact support, but it does not belong necessarily to $%
C_{0}^{\infty }\left( \mathbb{R}\right).$ In order to have a proper test
function for evaluating it in (\ref{def-all}), we need to convolute it with
a mollifier.

Let $\rho \in C_{0}^{\infty }\left( \mathbb{R}\right) $ be such that $\rho
\geq 0$ for $\left\vert x\right\vert \geq 1,$ $\rho \left( x\right) =\rho
\left( -x\right) $ and $\int_{\mathbb{R}}\rho \left( x\right) dx=1.$ For an
example of such function see e.g. \cite{barbu-book}.

We set $\rho _{\varepsilon }\left( x\right) =\dfrac{1}{\varepsilon }\rho
\left( \dfrac{x}{\varepsilon }\right) $ as a mollifier and we take the
regularization 
\begin{eqnarray*}
\overline{\varphi }_{\varepsilon }\left( x\right) &=&\int_{\mathbb{R}}%
\overline{\varphi }\left( x\right) \rho _{\varepsilon }\left( x-y\right)
dy\medskip \\
&=&\int_{\mathbb{R}}\overline{\varphi }\left( x-\varepsilon y\right) \rho
\left( y\right) dy,\quad \forall x\in \mathbb{R}\text{.}
\end{eqnarray*}

It is well-known that $\overline{\varphi }_{\varepsilon }\in C_{0}^{\infty
}\left( \mathbb{R}\right) $ and we can check that $\overline{\varphi }
_{\varepsilon }$ is also even. Indeed we have 
\begin{eqnarray*}
\overline{\varphi }_{\varepsilon }\left( -x\right)  &=&\int_{\mathbb{R}}%
\overline{\varphi }\left( -x-\varepsilon y\right) \rho \left( y\right) dy \\
&=&\int_{\mathbb{R}}\overline{\varphi }\left( x+\varepsilon y\right) \rho
\left( y\right) dy \\
&=&\int_{\mathbb{R}}\overline{\varphi }\left( x-\varepsilon y\right) \rho
\left( - y\right) dy=\overline{\varphi }_{\varepsilon }\left( x\right) .
\end{eqnarray*}

By \textit{Step I} we get 
\begin{equation*}
\int_{0}^{\infty }\overline{\varphi }_{\varepsilon }\left( x\right) v\left(
t,x\right) dx=\int_{0}^{\infty }\overline{\varphi }_{\varepsilon }\left(
x\right) u_{0}\left( x\right) dx+\frac{1}{2}\int_{0}^{t}\int_{0}^{\infty
}\left( \overline{\varphi }_{\varepsilon }\right) ^{\prime \prime }\left(
x\right) \eta _{v}\left( s,x\right) dxds.
\end{equation*}

Since $\overline{\varphi }_{\varepsilon }\left\vert _{\mathbb{R}_{+}}\right.
\in C_{0}^{\infty }\left( \mathbb{R}_{+}\right) $ we have 
\begin{equation*}
\left( \overline{\varphi }_{\varepsilon }\right) ^{\prime \prime }\left(
x\right) =\int_{\mathbb{R}}\overline{\varphi }^{\prime \prime }\left(
x-\varepsilon y\right) \rho \left( y\right) dy=\left( \overline{\varphi }%
^{\prime \prime }\right) _{\varepsilon }\left( x\right) ,\quad \forall x\in 
\mathbb{R}_{+};
\end{equation*}%
 then we can pass to the limit for $\varepsilon \rightarrow 0$ and
conclude the proof.
\end{proof}

\bigskip
\begin{corollary} \label{RWeakDistri}
Under Assumptions 1. and 2., the generalized solution of \eqref{equ-half}
is also a weak solution.
\end{corollary}
\begin{proof}
Let $(\bar u, \eta_{\bar u})$ be the solution of \eqref{equ-all}, i.e. on the real line
with initial condition $\bar u_0$ as in \eqref{uzerobar}.
Under Assumptions 1. and 2., Proposition 4.5 a) of \cite{BRR2} says that
for a.e. $t \in [0,T]$, $\eta_u(t, \cdot) \in H^1(\R)$ and
$ \int_{[0,T] \times \R} \eta_u(t,x)^2 dt dx < \infty.$
This implies that for every compact real interval $K$,  
$\eta_{\bar u}' \in L^1([0,T] \times K)$.  In particular 
$\eta_{\bar u} \in W^{1,1}_{\rm loc} ([0,T] \times \R)$.
By the proof of Proposition \ref{P5}, the solution
$(u, \eta_u)$ equals the restriction of $(2 \bar u, 2\bar \eta)$ to $\R_+$.   
This shows that  $\eta_{u})$ belongs to $W^{1,1}_{\rm loc} ([0,T] \times \R_+)$.
By Proposition \ref{PWeakDistri}, $u$ is also a weak solution.
\end{proof}

\begin{proposition}
\label{P5} Equation (\ref{equ-half}) has a unique generalized solution in
the sense of Definition \ref{def1}.
\end{proposition}

\begin{proof}
Existence has been the object of Proposition \ref{P4}, so we proceed now to
uniqueness. Let $\left( v,\eta _{v}\right) $ be a generalized solution of (%
\ref{equ-half}), i.e. for any $\varphi \in C_{0}^{\infty }\left( \mathbb{R}%
_{+}\right) $ such that $\varphi ^{\prime }\left( 0\right) =0$ we have 
\begin{equation}
\int_{0}^{\infty }\varphi \left( x\right) v\left( t,x\right)
dx=\int_{0}^{\infty }\varphi \left( x\right) u_{0}\left( x\right) dx+\frac{1%
}{2}\int_{0}^{t}\int_{0}^{\infty }\varphi ^{\prime \prime }\left( x\right)
\eta _{v}\left( s,x\right) dxds.  \label{def-half2}
\end{equation}%
We define $\bar{u}_{0}:\mathbb{R}\rightarrow \mathbb{R}$ by 
\begin{equation*}
\overline{u_{0}}\left( x\right) =\left\{ 
\begin{array}{cc}
\frac{1}{2}u_{0}\left( x\right)  & ,x\geq 0 \\ 
\frac{1}{2}u_{0}\left( -x\right)  & ,x<0%
\end{array}%
\right. 
\end{equation*}%
We extend the considered solution to $\left[ 0,T\right] \times \mathbb{R}$
setting 
\begin{equation*}
\overline{u}\left( t,x\right) =\left\{ 
\begin{array}{cc}
\frac{1}{2}v\left( t,x\right)  & ,x\geq 0 \\ 
\frac{1}{2}v\left( t,-x\right)  & ,x<0,%
\end{array}%
\right. 
\end{equation*}%
\begin{equation*}
\eta_{\bar u}\left( t,x\right) =\left\{ 
\begin{array}{cc}
\frac{1}{2} \eta _{v}\left( t,x\right)  & ,x\geq 0 \\ 
 \frac{1}{2} \eta _{v}\left( t,-x\right)  & ,x<0.%
\end{array}%
\right. 
\end{equation*}%
Obviously we  have 
\begin{equation*}
\eta_{\bar u}\left( t,x\right) \in \frac{1}{2} \beta \left( 2\overline{u}\left(
t,x\right) \right) = \bar \beta(\bar u(t,x)),\quad dt\otimes dx-a.e.~\left( t,x\right) \in \left[ 0,T%
\right] \times \mathbb{R}_{+}.
\end{equation*}%

We aim now at showing that $(\overline{u}, \overline{\eta_u})$  constructed
above, is a solution to equation (\ref{equ-all}) in order to use the
uniqueness for the equation on the whole line.
To this purpose we have to prove that 
\begin{equation*}
\int_{\mathbb{R}}\varphi \left( x\right) \overline{u}\left( t,x\right)
dx=\int_{\mathbb{R}}\varphi \left( x\right) \overline{u}_{0}\left( x\right)
dx+\frac{1}{2}\int_{0}^{t}\int_{\mathbb{R}}\varphi ^{\prime \prime }\left(
x\right)\eta_{\bar u}\left( s,x\right) dxds,
\end{equation*}%
for all $\varphi \in C_{0}^{\infty }\left( \mathbb{R}\right) $. 

Let us now fix $\varphi \in C_{0}^{\infty }\left( \mathbb{R}\right) $
and we can define $\psi :\mathbb{R}_{+}\rightarrow \mathbb{R}$ by $\psi
\left( x\right) =\varphi \left( x\right) +\varphi \left( -x\right) $.
Obviously $\psi \in C_{0}^{\infty }\left( \mathbb{R}_{+}\right) $ and $\psi
^{\prime }\left( 0\right) =0$, so it is a test function for (\ref{def-half2}%
). 
This gives 
\begin{equation*}
\int_{0}^{\infty }\psi \left( x\right) v\left( t,x\right)
dx=\int_{0}^{\infty }\psi \left( x\right) u_{0}\left( x\right) dx+\frac{1}{2}%
\int_{0}^{t}\int_{0}^{\infty }\psi ^{\prime \prime }\left( x\right) \eta
_{v}\left( s,x\right) dxds.
\end{equation*}%
The left-hand side above yields 
\begin{eqnarray*}
\int_{0}^{\infty }\psi \left( x\right) v\left( t,x\right) dx
&=&\int_{0}^{\infty }\left( \varphi \left( x\right) +\varphi \left(
-x\right) \right) v\left( t,x\right) dx \\
&=&\int_{0}^{\infty }\varphi \left( x\right) v\left( t,x\right)
dx+\int_{-\infty }^{0}\varphi \left( x\right) v\left( t,-x\right) dx \\
&=&\int_{\mathbb{R}}\varphi \left( x\right) 2\overline{u}\left( t,x\right)
dx,
\end{eqnarray*}%
by the obvious change of variable $x\mapsto -x$.
The same technique shows that 
\begin{equation*}
\int_{0}^{\infty }\psi \left( x\right) u_{0}\left( x\right) dx=\int_{\mathbb{%
R}}\varphi \left( x\right) 2\overline{u}_{0}\left( x\right) dx,
\end{equation*}%
and also 
\begin{equation*}
\int_{0}^{\infty }\psi ^{\prime \prime }\left( x\right) \eta _{v}\left(
s,x\right) dx= 2 \int_{\mathbb{R}}\varphi ^{\prime \prime }\left( x\right) 
{\eta }_{\bar u}\left( s,x\right) dx,\quad \text{for all }s\in \left[ 0,T%
\right] .
\end{equation*}

This implies
\begin{equation*}
2\int_{\mathbb{R}}\varphi \left( x\right) \overline{u}\left( t,x\right)
dx=2\int_{\mathbb{R}}\varphi \left( x\right) \overline{u_{0}}\left( x\right)
dx + 2 \int_{0}^{t}\int_{\mathbb{R}}\varphi ^{\prime \prime }\left(
x\right) {\eta }_{\bar u}\left( s,x\right) dxds,
\end{equation*}
where 
\begin{equation*}
\eta_{\bar u}\left( t,x\right) \in \bar \beta \left( \overline{u}\left(
t,x\right) \right) ,\quad a.e.~\left( t,x\right) \in \left[ 0,T\right]
\times \mathbb{R}_{+}
\end{equation*}
and we get that 
\begin{equation*}
\left\{ 
\begin{array}{l}
\partial_t \overline{u}\left( t,x\right) \in \frac{1}{2}\partial
_{xx}^{2}\left( \bar \beta \left(\overline{u}\left( t,x\right)
\right) \right) \quad :\left( t,x\right) \in \left( 0,T\right] \times 
\mathbb{R}, \\ 
\overline{u}\left( 0,x\right) =\overline{u_{0}}\left( x\right) .
\end{array}
\right. 
\end{equation*}

Finally, this shows that $\overline{u}$ is a solution to equation (\ref%
{equ-all}). As said at the beginning of the proof, the uniqueness for $v$  
 follows from uniqueness of \eqref{equ-all}.
\end{proof}

\begin{remark} \label{MassConv} 
 By Remark 1.6 of \cite{BRR2}, the solution $\bar u$ of \eqref{equ-all}
on the real line,
has mass conservation and it is always non-negative. 
Taking into account the construction of the generalized solution,
we also get a similar result for the generalized solution $v$ of \eqref{equ-half}.
This means the following.
\begin{enumerate}
\item $ v \ge 0$ a.e.
\item $\int_\R v(t,x) dx = 1, \ \forall t \in [0,1]$.
\end{enumerate}
This explains why the values of $\beta$ on $\R_-$ are not important, see 
Assumptions 1. and 2.
\end{remark}

\bigskip

\section{The  basic construction of the linear reflected diffusion}

\label{S3} \bigskip


In this section we are interested in the probabilistic representation of
equation (\ref{equ-half}), under Assumptions 1. and 2. from the introduction.
More precisely we aim at characterizing {\it all} the solutions of (\ref{equ-half}).

In this case we remind that we can write $\beta \left( u\right) =\Phi ^{2}\left(
u\right) u,$ for $\Phi $ being a non-negative   graph generated 
by a bounded function $\Phi$. 

We introduce now $\overline{\Phi }:\mathbb{R\rightarrow R}$ by $\overline{%
\Phi }\left( u\right) =\Phi \left( 2u\right).$ With this notation we have $%
\overline{\beta }\left( u\right) =u~\overline{\Phi }^{2}\left( u\right).$
Let $\bar{u}_{0}$ be as in \eqref{uzerobar}.

\bigskip

According to Theorem 4.4 in \cite{BRR1}, Theorem 5.4  in \cite{BRR2} taking into account the proof of 
Theorem 2.6 in \cite{BR},   equation \eqref{equ-all} admits a
probabilistic representation, in the sense that there is a solution $Y$ (in
law) of 
\begin{equation}  \label{stoch-all}
\left\{ 
\begin{array}{l}
Y_t \in Y_{0}+\int_{0}^{t}\bar \Phi \left( \overline{u}\left( s,Y_{s}\right)
\right) dW_{s}, t \in [0,T],  \\ 
\overline{u}\left( t,.\right) =\text{law density of }Y_{t}, t \in [0,T], \\
\bar u(0, \cdot) =\bar u_0,
\end{array}%
\right.  
\end{equation}%
for some Wiener process $\left( W_{s}\right) $ on some probability space. 

\bigskip
The precise meaning of the first line of \eqref{stoch-all} is the following:
\begin{equation} \label{EEE}
Y_t \in Y_{0}+\int_{0}^{t} \chi_{\overline{u}} \left( s,Y_{s}\right)
 dW_{s}, t \in [0,T],  \  
\chi_{\bar u}(t,y)  \in \bar \Phi(\bar u(t,y)), dt dy \ {\rm a.e.}
\end{equation}
In order not to loose the reader, without restriction of generality, 
we suppose in this section that $\Phi$ (and therefore $\bar \Phi$) is continuous on $\R_+$
so that we can write $\chi_u  = \Phi(u)$.
\begin{remark} \label{RPhiO}
We remind that we have defined $\Phi$ (and therefore $\bar \Phi$)  artificially at zero in the Introduction.
Observe that this has no influence in the probabilistic representation since 
$\int_0^T 1_{\{\bar u(s,Y_s) = 0\} }   \bar \Phi^2(\bar u(s,Y_s)) ds = 0 $ a.s. since
its expectation gives
$\int_0^T ds \int_\R 1_{\{\bar u(s,y) = 0\}} \bar u(s,y)  \bar \Phi^2(\bar 
u(s,y)) dy = 0.$ 
\end{remark}

Our purpose is to use this result in order to get a corresponding  one for the
porous media equation on the half-line with Neumann boundary condition. We
start with a preliminary result. \bigskip

\begin{lemma}
\label{P6} Let $(Y, \bar u)$ be 
 the solution of (\ref{stoch-all}). Let us denote by $\nu _{t},t\in
\lbrack 0,T]$, the marginal laws of $X_{t}=\left\vert Y_{t}\right\vert ,t\in
\lbrack 0,T]$. Then for all $t\in \lbrack 0,T]$, $\nu _{t}$ has a density
which is given by $v(t,\cdot )$ where $v=2~\overline{u}|_{[0,T]\times 
\mathbb{R}_{+}}$. 
\end{lemma}

\begin{proof}
We remind that $\overline{u}$ is the law density of $Y_{t}$ from equation (%
\ref{stoch-all}). Let $\varphi :\mathbb{R}_{+}\rightarrow \mathbb{R}$ be a
bounded Borel function and $t\in \lbrack 0,T]$. We notice that 
\begin{eqnarray*}
\mathbb{E}\left( \varphi \left( X_{t}\right) \right)  &=&\mathbb{E}\left(
\varphi \left( \left\vert Y_{t}\right\vert \right) \right) =\int_{\mathbb{R}%
}\varphi \left( \left\vert y\right\vert \right) \overline{u}\left(
t,y\right) dy= \\
&=&\int_{\mathbb{-\infty }}^{0}\varphi \left( -y\right) \overline{u}\left(
t,y\right) dy+\int_{0}^{\infty }\varphi \left( y\right) \overline{u}\left(
t,y\right) dy \\
&=&\int_{0}^{\infty }\varphi \left( y\right) 2~\overline{u}\left( t,y\right)
dy;
\end{eqnarray*}%
since $\bar u(t,\cdot)$ is even, this  concludes the proof.


\end{proof}

\bigskip 

Let $v$ be as in Lemma \ref{P6}. We continue by investigating the stochastic
equation solved by $X_{t}=\left\vert Y_{t}\right\vert ,t\in \lbrack 0,T]$
where $Y$ is the solution of (\ref{stoch-all}).

By It\^{o}-Tanaka formula (see e.g. \cite{revuz-yor} Theorem (1.2) Chapter
VI) we get 
\begin{equation}
X_{t}=\left\vert Y_{t}\right\vert =\left\vert Y_{0}\right\vert
+\int_{0}^{t}sgn\left( Y_{s}\right) \bar \Phi \left( \overline{u}\left(
s,Y_{s}\right) \right) dW_{s}+L_{t}^{Y}\left( 0\right) ,  \label{ito}
\end{equation}%
where 
\[
sgn(x)=\left\{ 
\begin{array}{cc}
1 & ,x>0 \\ 
-1 & ,x<0 \\ 
0 & ,x=0%
\end{array}%
\right. 
\]%
and $(L_{t}^{Y}\left( 0\right) )$ is the local time of the semimartingale $Y$
at zero. Since $Y$ is a local martingale, this is characterized
 by%
\begin{equation} \label{ito10}
L_{t}^{Y}\left( 0\right) =\frac{1}{2}~\underset{\varepsilon \rightarrow 0}{%
\lim }\int_{0}^{t} \frac{1}{\varepsilon}
 1_{\left\{ Y_{s}\in \left( -\varepsilon ,\varepsilon
\right) \right\} }d\left\langle Y\right\rangle _{s},
\end{equation}
see e.g.  Corollary (1.9) Chapter 6 \cite{revuz-yor}. We define 
\[
B_{t}^{1}=\int_{0}^{t}sgn\left( Y_{s}\right) dW_{s},
\]%
so that 
\[
\left\langle B^{1}\right\rangle _{t}=\int_{0}^{t}(sgn)^{2}\left( Y_{s}\right)
ds=\int_{0}^{t}1_{\left\{ Y_{s}\neq 0\right\} }ds.
\]%
Possibly enlarging the probability space, let $B^{2}$ be an independent
Brownian motion of $W$. We set 
\begin{equation}
B:=B^{1}+\int_{0}^{\cdot }1_{\left\{ Y_{s}=0\right\} }dB_{s}^{2}.
\label{eB12}
\end{equation}%
Since $\langle W,B^{2}\rangle _{t}=0$, it follows that 
\[
\langle B\rangle _{t}=\int_{0}^{t}1_{\left\{ Y_{s}\neq 0\right\}
}ds+\int_{0}^{t}1_{\left\{ Y_{s} = 0\right\} }ds=t,\forall t\in \lbrack
0,T].
\]%
By L\'{e}vy characterization theorem of the Brownian motion, it follows that 
$B$ is a Wiener process.

\bigskip By (\ref{ito}), since $L^{Y}$ is a bounded variation process, we
show that 
\begin{equation}
\langle X\rangle _{t}=\int_{0}^{t}1_{\left\{ Y_{s}\neq 0\right\} }\overline{%
\Phi }^{2}(\bar{u}(s,Y_{s}))ds=\int_{0}^{t}1_{\left\{ Y_{s}\neq 0\right\}
}\Phi ^{2}(2\bar{u}(s,Y_{s}))ds  \label{ito1}
\end{equation}%
\begin{equation*}
=\int_{0}^{t}1_{\left\{ X_{s}\neq 0\right\} }\Phi ^{2}(v(s,X_{s}))ds,
\end{equation*}%
because $\bar{u}(s,y)=\dfrac{v}{2}(s,|y|),~s\in \lbrack 0,T],~y\in \mathbb{R}
$. Coming back to \eqref{ito10}, again by Corollary (1.9) Chapter VI of \cite{revuz-yor} and  (\ref{stoch-all}), we get
\begin{eqnarray*}
L_{t}^{Y}\left( 0\right)  &=&\underset{\varepsilon \rightarrow 0}{\lim }%
\frac{1}{2\varepsilon }\int_{0}^{t}1_{\left\{ Y_{s}\in \left( -\varepsilon
,\varepsilon \right) \right\} }d\left\langle Y\right\rangle _{s}\medskip  \\
&=&\underset{\varepsilon \rightarrow 0}{\lim }\frac{1}{2\varepsilon }%
\int_{0}^{t}1_{\left\{ Y_{s}\in \left( -\varepsilon ,\varepsilon \right)
\right\} }\overline{\Phi }^{2}\left( \overline{u}\left( s,Y_{s}\right)
\right) ds\medskip  \\
&=&\underset{\varepsilon \rightarrow 0}{\lim \frac{1}{2\varepsilon }}%
\int_{0}^{t}1_{\left\{ \left\vert Y_{s}\right\vert \in \left[ 0,\varepsilon
\right) \right\} }\Phi ^{2}\left( 2\overline{u}\left( s,Y_{s}\right) \right)
ds\medskip  \\
&=&\underset{\varepsilon \rightarrow 0}{\lim \frac{1}{2\varepsilon }}%
\int_{0}^{t}1_{\left\{ X_{s}\in \left[ 0,\varepsilon \right) \right\} }\Phi
^{2}\left( v\left( s,X_{s}\right) \right) ds\medskip  \\
&=&\underset{\varepsilon \rightarrow 0}{\lim \frac{1}{2\varepsilon }}%
\int_{0}^{t}1_{\left\{ X_{s}\in \left( 0,\varepsilon \right) \right\} }\Phi
^{2}\left( v\left( s,X_{s}\right) \right) ds\medskip  \\
&&\quad \quad \quad +\underset{\varepsilon \rightarrow 0}{\lim }\frac{1}{%
2\varepsilon }\int_{0}^{t}1_{\left\{ X_{s}=0\right\} }\Phi ^{2}\left(
v\left( s,X_{s}\right) \right) ds.
\end{eqnarray*}
because of Lemla \ref{P6}.
By (\ref{ito1}) it follows that 
\begin{eqnarray} \label{AAA5}
L_{t}^{Y}\left( 0\right)  &=&\underset{\varepsilon \rightarrow 0}{\lim }%
\frac{1}{2\varepsilon }\int_{0}^{t}1_{\left\{ X_{s}\in \left[ 0,\varepsilon
\right) \right\} }d\left\langle X\right\rangle _{s}  \nonumber \\
&& \\ 
&&\quad \quad \quad +\underset{\varepsilon \rightarrow 0}{\lim }\frac{1}{%
2\varepsilon }\int_{0}^{t}1_{\left\{ X_{s}=0\right\} }\Phi ^{2}\left(
v\left( s,X_{s}\right) \right) ds.  \nonumber
\end{eqnarray}

Now
\begin{equation}
\int_{0}^{t}1_{\left\{ X_{s}=0\right\} }\Phi ^{2}\left( v\left(
s,X_{s}\right) \right) ds=0,  \label{AAA}
\end{equation}%
because the expectation of the non-negative left-hand side gives 
\begin{equation*}
\int_{0}^{t}\int_{\{0\}}\Phi ^{2}\left( v(s,y)\right) v(s,y)dyds=0.
\end{equation*}

By (\ref{AAA}), \eqref{AAA5}  and taking also into account again 
 Corollary (1.9) 
Chapter VI of \cite%
{revuz-yor} we obtain that%
\begin{equation*}
L_{t}^{Y}\left( 0\right) =\frac{1}{2}L_{t}^{X}\left( 0\right) .
\end{equation*}

Going back to (\ref{ito}) we obtain 
\begin{equation}
\left\{ 
\begin{array}{l}
X_{t}=X_{0}+\int_{0}^{t}\Phi \left( v\left( s,X_{s}\right) \right)
dB_{s}^{1}+\dfrac{1}{2}L_{t}^{X}\left( 0\right), \\ 
v\left( t,.\right) =\text{law density of }X_{t}\text{, }t\in \left[ 0,T%
\right],  \\ 
v(0,.)=u_{0}.%
\end{array}%
\right.   \label{stoch-half}
\end{equation}%

Taking into account (\ref{eB12}) we obtain 
\begin{equation}
\left\{ 
\begin{array}{l}
X_{t}=X_{0}+\int_{0}^{t}\Phi \left( v\left( s,X_{s}\right) \right) dB_{s}+%
\dfrac{1}{2}L_{t}^{X}\left( 0\right), \\ 
v\left( t,.\right) =\text{law density of }X_{t},\\
v(0, \cdot) = u_0.
\end{array}%
\right.   \label{stoch-half1}
\end{equation}

This happens because $\int_{0}^{t}\Phi (v(s,X_{s}))dB_{s}^{2}=0$. In fact
its expectation gives $\E\left( 
\int_{0}^{T} \Phi ^{2}(v(s,X_{s}))ds\right)=0$   since for each 
$s\in \lbrack 0,T],$ $X_{s}$ has a density. The proof is now concluded.

\section{The probabilistic representation}
\label{S4}



A solution of (\ref{stoch-half1}) is the candidate for a
probabilistic representation of a solution to (\ref{stoch-half}). The
proposition below gives an illustration of this, even though this constitutes
the {\it easy} part of the probabilistic representation.

We suppose again the validity of Assumptions 1. and 2.

\begin{proposition}
\label{P7} Let $X$ be a solution in law  of%
\begin{equation}
\left\{ 
\begin{array}{l}
X_{t}=X_{0}+\int_{0}^{t}\chi _{v}\left( s,X_{s}\right) dB_{s}+\dfrac{1}{2}%
L_{t}^{X}\left( 0\right) , \\ 
v\left( t,.\right) =\text{law density of }X_{t}\text{, }t\in \left[ 0,T%
\right],  \\ 
\chi _{v}\left( s,x\right) \in \Phi \left( v\left( s,x\right) \right) ,\quad
ds\otimes dx~a.e.\\
X \ge 0,\\
v(0, \cdot) = u_0.%
\end{array}%
\right.   \label{stoch-half2}
\end{equation}

Then $v$ is the generalized solution  in the sense of Definition \ref{def1}
 to equation (\ref{equ-half}).
\end{proposition}

\begin{remark}\label{R77}

\begin{itemize}
\item[i)] Equation (\ref{stoch-half2}) is formalized by 
\begin{equation} \label{stoch-half3}
\left\{ 
\begin{array}{l}
X_{t}\in X_{0}+\int_{0}^{t}\Phi \left( v\left( s,X_{s}\right) \right) dB_{s}+%
\dfrac{1}{2}L_{t}^{X}\left( 0\right) , \\ 
v\left( t,.\right) =\text{law density of }X_{t}\text{, }t\in \left[ 0,T%
\right],
v(0,\cdot) = 0, \\
X \ge 0. \\
 \end{array}%
\right. 
\end{equation}

\item[ii)] If $\Phi $ is continuous    (\ref{stoch-half3}). coincide with (\ref{stoch-half1}).

\item[iii)] According to Corollary \ref{RWeakDistri} $v$ is also the weak solution 
of  (\ref{stoch-half}) in the sense of Definition \ref{DWeak}.

\end{itemize}
\end{remark}

\begin{proof}
Let $\varphi \in C_{0}^{\infty }\left( \mathbb{R}_{+}\right) $ be such that $%
\varphi ^{\prime }\left( 0\right) =0.$
By  It\^{o} formula we obtain 
\begin{eqnarray}
\varphi \left( X_{t}\right)  &=&\varphi \left( X_{0}\right)
+\int_{0}^{t}\varphi ^{\prime }\left( X_{s}\right) \chi _{v}\left(
s,X_{s}\right) dB_{s}  \label{ito2} \\
&&\quad +\frac{1}{2}\int_{0}^{t}\varphi ^{\prime \prime }\left( X_{s}\right)
\chi _{v}^{2}\left( s,X_{s}\right) ds+\int_{0}^{t}\varphi ^{\prime }\left(
X_{s}\right) dL_{s}^{X}\left( 0\right) .  \nonumber
\end{eqnarray}

The last term above gives
\[
\int_{0}^{t}\varphi ^{\prime }\left( X_{s}\right) dL_{s}^{X}\left( 0\right)
=\int_{0}^{t}1_{\left\{ X_{s}=0\right\} }\varphi ^{\prime }\left(
X_{s}\right) dL_{s}^{X}\left( 0\right) +\int_{0}^{t}1_{\left\{ X_{s}\neq
0\right\} }\varphi ^{\prime }\left( X_{s}\right) dL_{s}^{X}\left( 0\right) .
\]

The first integral of the right-hand side equals 
\[
\varphi ^{\prime }\left( 0\right) \int_{0}^{t}1_{\left\{ X_{s}=0\right\}
}dL_{s}^{X}\left( 0\right) =0,
\]%
by the assumptions on $\varphi .$

The absolute value of the second integral is bounded by 
\[
\left\Vert \varphi ^{\prime }\right\Vert _{\infty }\int_{0}^{t}1_{\left\{
X_{s}\neq 0\right\} }dL_{s}^{X}\left( 0\right) =0,
\]%
by Theorem 7.1 from \cite{KARSH}. So
we can conclude that $\int_{0}^{t}\varphi ^{\prime }\left( X_{s}\right)
dL_{s}^{X}\left( 0\right) =0.$

Now, by taking the expectation in (\ref{ito2}), for $t\in \left[ 0,T\right] $,
we obtain 
\begin{eqnarray*}
\int_{0}^{\infty }\varphi \left( x\right) v\left( t,x\right) dx
&=&\int_{0}^{\infty }\varphi \left( x\right) u_{0}\left( x\right) dx \\
&&+\frac{1}{2}\int_{0}^{t}\int_{0}^{\infty }\varphi ^{\prime \prime }\left(
x\right) v\left( s,x\right) \chi _{v}^{2}\left( s,x\right) dxds.
\end{eqnarray*}

Since $\beta \left( u\right) =u~\Phi ^{2}\left( u\right) ,$ setting $\eta
_{v}\left( s,x\right) =v\left( s,x\right) \chi _{v}^{2}\left( s,x\right) ,$
we have $\eta _{v}\left( s,x\right) \in \beta \left( v\left( s,x\right)
\right) ,$ $ds\otimes dx~a.e.$ and $v$ is a generalized solution. We remind
that $v$ is necessarily the unique solution of (\ref{equ-half}) and we can
conclude the proof.
\end{proof}

\bigskip

\begin{remark} \label{R10}
Any solution to equation 
\[
\left\{ 
\begin{array}{l}
dX_{t}\in \Phi \left( v\left( t,X_{t}\right) \right) dB_{t}+\frac{1}{2}%
dL_{t}^{X}\left( 0\right) , \\ 
v\left( t,.\right) =\text{law density of }X_{t}\text{, }t\in \left[ 0,T%
\right] , \\ 
v(0,.)=u_{0},%
\end{array}%
\right. 
\]%
solves also the reflecting Skorohod problem 
\begin{equation}
\left\{ 
\begin{array}{l}
dX_{t}\in \Phi \left( v\left( t,X_{t}\right) \right) dB_{t}+dK_{t}, \\ 
K\text{ is increasing}, \\ 
 X \ge 0, \int_{0}^{T}X_{s}dK_{s}=0,
\end{array}%
\right.   \label{refl}
\end{equation}%
with $L_{t}^{X}\left( 0\right) =2K_{t}.$

Indeed, $(L_{t}^{X}\left( 0\right)) $ is an increasing process, being a limit
of increasing processes (see Corollary 1.9 Chapter VI of \cite{revuz-yor}) and $%
\int_{0}^{t}1_{\left\{ X_{s}=0\right\} }dL_{s}^{X}\left( 0\right) =0$ by
Theorem 7.1 from \cite{KARSH}.

In particular showing the existence of a solution to \eqref{stoch-half3}, also establishes the  existence for equation (\ref{refl}).
\end{remark}

The consideration of the first part of Section 3 allows to state effectively the
existence of solutions for (\ref{stoch-half2}) which consists in a sort of
converse statement of Proposition \ref{P7}. 

\begin{theorem} \label{T11}
Under Assumptions 1 and 2 there is a solution in law of (\ref{stoch-half2}).
\end{theorem}

\begin{proof}




The proof is the object of Section \ref{S3} with some obvious
adaptations to the case when $\Phi$ was not continuous
and it has to be associated with a graph.

\end{proof}
\begin{remark} \label{R12} The conclusion of Theorem \ref{T11} is valid under
other settings of similar hypotheses (but different) as those in 
 Assumptions 1 and 2. For instance under the one of the following assumptions, 
there is a solution in law to \eqref{stoch-half2}.
\begin{itemize}
\item Suppose $\Phi$ to be continuous on $\R_+$ and Assumption 1. ii) 
 is valid. Then the conclusion of Theorem \ref{T11} 
still holds if in Assumption 1. we replace $u_0 \in L^1(\R) \cap L^\infty(\R)  $ with
$u_0 \in L^1(\R)$.
In particular,  there is a solution in law to    to \eqref{stoch-half2}.
Indeed by Theorem 2.6 of \cite{BR}, a probabilistic representation on 
the whole line takes place. This constitutes the replacement tool to Theorem 4.4 of \cite{BRR1} and 
of   Theorem 5.4 of \cite{BRR2}, mentioned in the lines before \eqref{stoch-all}. 
\item  Item iii) in Assumption 1. is somehow technical and it can be often relaxed
for instance when $\beta$ is non-degenerate and $\Phi$ is smooth. 
Typically, when $\Phi$ is of of class $C^3$, Lipschitz, $\beta$ is non-degenerate
and $u_0$ is absolutely continuous with derivative in $H^{2 + \alpha}$ for some $0 < \alpha < 1$.
Then there is a (even a strong) solution of   \eqref{stoch-all},
see for instance Proposition 2.2 in \cite{jourdainmeleard}. 
As for the  previous item, by the same proof as for Theorem \ref{T11},
there will be a solution  to \eqref{stoch-half2}.
\end{itemize}

\end{remark}

\section{Appendix}

{\bf Proof of Proposition \ref{PWeakDistri}.}

Let $ \varphi \in C_{0}^{\infty} (\mathbb{R}_{+})$. Operating by integration by parts,  (\ref{weak}) it is equivalent to 
\begin{eqnarray}
\int_{0}^{\infty }\varphi \left( x\right) u\left( t,x\right) dx
&=&\int_{0}^{\infty }\varphi \left( x\right) u_{0}\left( x\right) dx
\label{equiv} \\
&&-\int_{0}^{t}\varphi ^{\prime }\left( 0\right) \eta _{u}\left( s,0\right)
ds+\int_{0}^{t}\int_{0}^{\infty }\varphi ^{\prime \prime }\left( x\right)
\eta _{u}\left( s,x\right) dx ds.  \nonumber 
\end{eqnarray}

We observe that $\left( u,\eta _{u}\right) $ is a generalized solution if (\ref{equiv}) holds for every $\varphi \in C_{0}^{\infty }\left( \mathbb{R}_{+}\right) $ 
such that $\varphi ^{\prime}\left( 0\right) =0$.

We can easily see that if $\left( u,\eta _{u}\right) $ is a weak solution
then it is obviously also a generalized solution. 

In the other sense we assume that $(u,\eta _{u})$ be a generalized solution.
To show that it is also a weak solution we need to show that (\ref{weak})
holds for every $\varphi \in C_{0}^{\infty }\left( \mathbb{R}_{+}\right) .$
We know in fact that (\ref{equiv}) holds for every $\varphi \in C_{0}^{\infty
}\left( \mathbb{R}_{+}\right) $ such that $\varphi ^{\prime }\left( 0\right)
=0.$

Let $\varepsilon >0.$ We consider a sequence $\chi _{\varepsilon }\left(
x\right) :\mathbb{R}_{+}\rightarrow \left[ 0,1\right] $ smooth such that 
\[
\chi _{\varepsilon }\left( x\right) =\left\{ 
\begin{array}{ll}
0 & :x\in \left[ 0,\varepsilon \right)  \\ 
1 & :x>2\varepsilon .%
\end{array}%
\right. 
\]

We set 
\[
\varphi _{\varepsilon }\left( x\right) =\int_{0}^{x}\chi _{\varepsilon
}\left( y\right) \varphi ^{\prime }\left( y\right) dy + \varphi \left(
0\right) +c\left( \varepsilon \right) 
\]%
where%
\[
c\left( \varepsilon \right) =\int_{0}^{\infty }\left( 1-\chi _{\varepsilon
}\right) \left( y\right) \varphi ^{\prime }\left( y\right) dy.
\]

We note that $\varphi _{\varepsilon }$ is obviously smooth. Moreover it has
compact support since for $x>\max ({\rm  supp} \varphi )$ we have%
\[
\varphi _{\varepsilon }\left( x\right) =\int_{0}^{\infty }\chi _{\varepsilon
}\left( y\right) \varphi ^{\prime }\left( y\right) dy+\varphi (0)+c\left(
\varepsilon \right) =0.
\]

By assumption, 
(\ref{equiv}) holds with $\varphi $ replaced by $\varphi _{\varepsilon
}.$ Obviously $\varphi _{\varepsilon }\rightarrow \varphi $ pointwise and $%
\varphi _{\varepsilon }^{\prime }\rightarrow \varphi ^{\prime }$ in $%
L^{p}\left( \mathbb{R}\right) ,$ for all $p\geq 1.$

Since $\varphi _{\varepsilon }^{\prime }\left( 0\right) =0,$ (\ref{weak})
holds for $\varphi $ replaced by $\varphi _{\varepsilon }$, i.e. 
\begin{eqnarray}
\int_{0}^{\infty }\varphi _{\varepsilon }\left( x\right) u\left( t,x\right)
dx &=&\int_{0}^{\infty }\varphi _{\varepsilon }\left( x\right) u_{0}\left(
x\right) dx  \label{weak-aprox} \\
&&-\int_{0}^{t}\int_{0}^{\infty }\varphi _{\varepsilon }^{\prime }\left(
x\right) \eta _{u}^{\prime }\left( s,x\right) dxds.  \nonumber
\end{eqnarray}

We remark that $c(\varepsilon )\rightarrow 0$ for $\varepsilon \rightarrow 0$ and 
\[
\left\vert \varphi _{\varepsilon }\left( x\right) \right\vert \leq
\int_{0}^{\infty }\left\vert \varphi ^{\prime }\left( y\right) \right\vert
dy+\varphi \left( 0\right) + \underset{\varepsilon}{\sup } \ c(\varepsilon ).
\]
Let $t \in [0,T]$.
By the dominated convergence theorem the left-hand side of (\ref{weak-aprox})
(respectively the first integral in the right-hand side), converges to $%
\int_{0}^{\infty }\varphi \left( x\right) u\left( t,x\right) dx$ (respectively 
$\int_{0}^{\infty }\varphi \left( x\right) u_{0}\left( x\right) dx).$ 

We observe 
\[
\left\vert \varphi _{\varepsilon }^{\prime }\left( x\right) \right\vert \leq
\left\vert \varphi ^{\prime }\left( x\right) \right\vert \mathbf{1}_{I}(x),
\]%
$I$ being a compact interval including the support of $\varphi .$

Again by dominated convergence theorem the second integral in the right-hand
side of (\ref{weak-aprox}), converges to $\int_{0}^{t}\int_{0}^{\infty
}\varphi ^{\prime }\left( x\right) \eta _{u}^{\prime }\left( s,x\right) dxds.
$ This shows (\ref{weak}) for (\ref{equ-half}) and conclude the proof.


{\bf ACKNOWLEDGEMENTS.}
The work of the second name author
was supported by the
ANR Project  MASTERIE 2010 BLAN--0121--01.
The first part of the work was done during the stay of the authors  
at the Bernoulli Center of the EPFL Lausanne.

\bibliographystyle{plain}
\bibliography{ICFR_Bibliography}

\end{document}